\documentclass[11pt]{amsart}
\usepackage{amsmath}
\usepackage{amsmath,amsfonts,amssymb,amsthm}
\usepackage[active]{srcltx}

\newtheorem{theorem}{Theorem}

\newtheorem{remark}{Remark}
\newtheorem{definition}{Definition}
\newtheorem{corollary}{Corollary}

\newtheorem*{Sahakian}{Lemma S}
\newtheorem*{Goginava-Sahakian}{Theorem GS}

\newtheorem*{Sah}{Theorem WS}
\newtheorem*{Bakh}{Theorem B}

\begin{document}
\author{Ushangi Goginava and Artur Sahakian}
\title[Convergence of multiple Fourier series]{Convergence of multiple
Fourier series of functions of bounded generalized variation}
\address{U. Goginava, Department of Mathematics, Faculty of Exact and
Natural Sciences, Iv. Javakhishvili Tbilisi State University, Chavcha\-vadze
str. 1, Tbilisi 0128, Georgia}
\email{zazagoginava@gmail.com}
\address{A. Sahakian, Yerevan State University, Faculty of Mathematics and
Mechanics, Alex Manoukian str. 1, Yerevan 0025, Armenia}
\email{sart@ysu.am}
\maketitle

\begin{abstract}
The paper introduces a new concept of $\Lambda $-variation of multivariable
functions and investigates its connection with the convergence of
multidimensional Fourier series
\end{abstract}

\footnotetext{%
}

\section{Classes of Functions of Bounded Generalized Variation}

In 1881 Jordan \cite{Jo} introduced a class of functions of bounded
variation and applied it to the theory of Fourier series. Hereafter this
notion was generalized by many authors (quadratic variation, $\Phi $%
-variation, $\Lambda $-variation ets., see \cite{Ch,Wi,W,M}). In two
dimensional case the class BV of functions of bounded variation was
introduced by Hardy \cite{Ha}.

For an interval $T=[a,b]\subset R$ we denote by $T^d=[a,b]^d$ the
d-dimensional cube in $R^d$.

Consider a function $f\left( x\right) $ defined on $T^{d}$ and a collection
of intervals
\begin{equation*}
J^{k}=\left( a^{k},b^{k}\right) \subset T,\qquad k=1,2,\ldots d.
\end{equation*}
For $d=1$ we set
\begin{equation*}
f\left( J^{1}\right) :=f\left( b^{1}\right) -f\left( a^{1}\right) .
\end{equation*}%
If for any function of $d-1$ variables the expression $f\left( J^{1}\times
\cdots \times J^{d-1}\right) $ is already defined, then for a function of $d$
variables the \textit{mixed difference} is defined as follows:
\begin{equation*}
f\left( J^{1}\times \cdots \times J^{d}\right) :=f\left( J^{1}\times \cdots
\times J^{d-1},b^{d}\right) -f\left( J^{1}\times \cdots \times
J^{d-1},a^{d}\right) .
\end{equation*}

Let $E=\{I_{k}\}$ be a collection of nonoverlapping intervals from $T$
ordered in arbitrary way and let $\Omega=\Omega(T) $ be the set of all such
collections $E$. We denote by $\Omega _{n}=\Omega_{n}(T)$ set of all
collections of $n$ nonoverlapping intervals $I_{k}\subset T.$

For sequences of positive numbers
\begin{equation*}
\Lambda ^{j}=\{\lambda _{n}^{j}\}_{n=1}^{\infty },\quad
\lim_{n\to\infty}\lambda^j_n=\infty,\quad j=1,2,\ldots ,d,
\end{equation*}
and for a function $f(x)$, $x=(x_1,\ldots,x_d)\in T^d$ the $\left( \Lambda
^{1},\ldots ,\Lambda ^{d}\right) $-\textit{variation of $f$ with respect to
the index set }$D:=\{1,2,...,d\}$ is defined as follows:
\begin{equation*}
\left\{ \Lambda ^{1},\ldots ,\Lambda ^{d}\right\} V^{D}\left( f,T^{d}\right)
:=\sup\limits_{\{I_{i_{j}}^{j}\}\in \Omega }\ \sum\limits_{i_{1},...,i_{d}}%
\frac{\left\vert f\left( I_{i_{1}}^{1}\times \cdots \times
I_{i_{d}}^{d}\right) \right\vert }{\lambda^1 _{i_{1}}\cdots \lambda^d
_{i_{d}}}.
\end{equation*}

For an index set $\alpha =\{j_{1},...,j_{p}\}\subset D$ and any $x=\left(
x_{1},...,x_{d}\right) \in R^{d}$ we set ${\widetilde{\alpha }}:=D\setminus
\alpha $ and denote by $x_{\alpha }$ the vector of $R^{p}$ consisting of
components $x_{j},j\in \alpha $, i.e.
\begin{equation*}
x_{\alpha }=\left( x_{j_{1}},...,x_{j_{p}}\right) \in R^{p}.
\end{equation*}

By
\begin{equation*}
\left\{ \Lambda ^{j_{1}},...,\Lambda ^{j_{p}}\right\} V^{{\alpha }}\left(
f,x_{\widetilde{\alpha }},T^d\right) \quad \text {and}\quad f\left(
I_{i_{j_{1}}}^{1}\times \cdots \times I_{i_{j_{p}}}^{p},x_{\widetilde{\alpha
}}\right)
\end{equation*}
we denote respectively the $\left( \Lambda ^{j_{1}},...,\Lambda
^{j_{p}}\right) $-variation over the $p$-dimensional cube $T^{p}$ and mixed
difference of $f$ as a function of variables $x_{^{j_{1}}},...,x_{j_{p}}$
with fixed values $x_{^{\widetilde{\alpha }}}$ of other variables. The $%
\left( \Lambda ^{j_{1}},...,\Lambda ^{j_{p}}\right) $\textit{-variation of }$%
f$\textit{\ with respect to the index set} ${\alpha }$ is defined as
follows:
\begin{equation*}
\left\{ \Lambda ^{j_{1}},...,\Lambda ^{j_{p}}\right\} V^{{\alpha }}\left(
f,T^{p}\right) =\sup\limits_{x_{^{{\widetilde{\alpha }}}}\in T^{d-p}}\left\{{%
\Lambda ^{j_{1}},...,\Lambda ^{j_{p}}}\right\}V^{{\alpha }}\left( f,x_{^{%
\widetilde{\alpha }}},T^{d}\right) .
\end{equation*}

\begin{definition}
We say that the function $f$ has total Bounded $\left( \Lambda
^{1},...,\Lambda ^{d}\right) $-variation on $T^{d}$ and write $f\in \left\{
\Lambda ^{1},...,\Lambda ^{d}\right\}BV\left( T^{d}\right) $, if
\begin{equation*}
\left\{ \Lambda ^{1},...,\Lambda ^{d}\right\}
V(f,T^{d}):=\sum\limits_{\alpha \subset D}\left\{ \Lambda ^{1},...,\Lambda
^{d}\right\}V^{{\alpha }}\left( f,T^{d}\right) <\infty .
\end{equation*}
\end{definition}

\begin{definition}
We say that the function $f$ is continuous in $\left( \Lambda
^{1},...,\Lambda ^{d}\right) $-variation on $T^{d}$ and write $f\in C\left\{
\Lambda ^{1},...,\Lambda ^{d}\right\} V\left( T^{d}\right) $, if%
\begin{equation*}
\lim\limits_{n\rightarrow \infty }\left\{ \Lambda ^{j_{1}},...,\Lambda
^{j_{k-1}},\Lambda _{n}^{j_{k}},\Lambda ^{j_{k+1}},...,\Lambda
^{j_{p}}\right\} V^{{\alpha }}\left( f,T^{d}\right) =0,\qquad k=1,2,\ldots ,p
\end{equation*}%
for any $\alpha \subset D,\ \alpha :=\{j_{1},...,j_{p}\}$, where $\Lambda
_{n}^{j_{k}}:=\left\{ \lambda _{s}^{j_{k}}\right\} _{s=n}^{\infty }$.
\end{definition}

\begin{definition}
We say that the function $f$ has Bounded Partial $\left( \Lambda
^{1},...,\Lambda ^{d}\right) $-variation and write $f\in P\left\{ \Lambda
^{1},...,\Lambda ^{d}\right\} BV\left( T^{d}\right) $ if
\begin{equation*}
P\left\{ \Lambda ^{1},...,\Lambda ^{d}\right\}
V(f,T^{d}):=\sum\limits_{i=1}^{d}\Lambda ^{i}V^{\{i\}}\left( f,T^{d}\right)
<\infty .
\end{equation*}
\end{definition}

In the case $\Lambda ^{1}=\cdots =\Lambda ^{d}=\Lambda $ we set
\begin{eqnarray*}
\Lambda BV(T^{d}):= &&\{\Lambda ^{1},...,\Lambda ^{d}\}BV(T^{d}), \\
C\Lambda V(T^{d}):= &&C\{\Lambda ^{1},...,\Lambda ^{d}\}V(T^{d}), \\
P\Lambda BV(T^{d}):= &&P\{\Lambda ^{1},...,\Lambda ^{d}\}BV(T^{d}).
\end{eqnarray*}%
If $\lambda _{n}\equiv 1$ (or if $0<c<\lambda _{n}<C<\infty ,\ n=1,2,\ldots $%
) the classes $\Lambda BV$ and $P\Lambda BV$ coincide with the Hardy class $%
BV$ and $PBV$ respectively. Hence it is reasonable to assume that $\lambda
_{n}\rightarrow \infty $ .

When $\lambda _{n}=n$ for all $n=1,2\ldots $ we say \textit{Harmonic
Variation} instead of $\Lambda $-variation and write $H$ instead of $%
\Lambda, i.e. $ $HBV$, $PHBV$, $CHV$, ets.

For two variable functions Dyachenko and Waterman \cite{DW} introduced
another class of functions of generalized bounded variation.

Denoting by $\Gamma $ the the set of finite collections of nonoverlapping
rectangles $A_{k}:=\left[ \alpha _{k},\beta _{k}\right] \times \left[ \gamma
_{k},\delta _{k}\right] \subset T^{2}$, for a function $f(x,y), \ x,y\in T$,
we set
\begin{equation*}
\Lambda ^{\ast }V\left( f,T^{2}\right) :=\sup_{\{A_{k}\}\in \Gamma
}\sum\limits_{k}\frac{\left\vert f\left( A_{k}\right) \right\vert }{\lambda
_{k}}.
\end{equation*}

\begin{definition}[Dyachenko, Waterman]
We say that $f\in \Lambda ^{\ast }BV\left( T^{2}\right) $ if
\begin{equation*}
\Lambda V(f,T^{2}):=\Lambda V_{1}(f,T^{2})+\Lambda V_{2}(f,T^{2})+\Lambda
^{\ast }V\left( f,T^{2}\right) <\infty .
\end{equation*}
\end{definition}

In this paper we introduce a new classes of functions of generalized bounded
variation and investigate the convergence of Fourier series of function of
that classes.

For the sequence $\Lambda =\{\lambda _{n}\}_{n=1}^{\infty }$ we denote%
\begin{equation*}
\Lambda ^{\#}V_{s}\left( f,T^{d}\right) :=\sup\limits_{\left\{ x^{i}{\left\{
s\right\} }\right\} \subset T^{d-1}}\sup\limits_{\left\{ I_{i}^{s}\right\}
\in \Omega }\sum\limits_{i}\frac{\left\vert f\left( I_{i}^{s},x^{i}{\left\{
s\right\} }\right) \right\vert }{\lambda _{i}},
\end{equation*}%
where
\begin{equation}
x^{i}{\left\{ s\right\} }:=\left( x_{1}^{i},\ldots
,x_{s-1}^{i},x_{s+1}^{i},\ldots ,x_{d}^{i}\right) \quad \text{for}\quad
x^{i}:=\left( x_{1}^{i},\ldots ,x_{d}^{i}\right) .  \label{xis}
\end{equation}

\begin{definition}
\label{def5} We say that the function $f$ belongs to the class $\Lambda
^{\#}BV\left( T^{d}\right) $, if%
\begin{equation*}
\Lambda ^{\#}V\left( f,T^{d}\right) :=\sum\limits_{s=1}^{d}\Lambda
^{\#}V_{s}\left( f,T^{d}\right) <\infty .
\end{equation*}
\end{definition}

The notion of $\Lambda $-variation was introduced by Waterman \cite{W} in
one dimensional case, by Sahakian \cite{Saha} in two dimensional case and by
Sablin \cite{Sab} in the case of higher dimensions. The notion of bounded
partial variation (class $PBV$) was introduced by Goginava in \cite{GoJAT}.
These classes of functions of generalized bounded variation play an
important role in the theory Fourier series.

\begin{remark}
It is not hard to see that $\Lambda ^{\#}BV\left( T^{d}\right)\subset
P\Lambda BV\left(T^{d}\right)$ for any $d>1$ and $\Lambda ^{\ast}BV\left(
T^{2}\right)\subset \Lambda ^{\#}BV\left( T^{2}\right) $.
\end{remark}

We prove that the following theorem is true.

\begin{theorem}
\label{t1} Let $d\geq 2$ and $T=(t_{1},t_{2})\subset R$. If
\begin{equation}
\Lambda =\left\{ \lambda _{n}\right\} \quad \text{with}\quad \lambda _{n}=%
\frac{n}{\log ^{d-1}\left( n+1\right) },\quad n=1,2,\ldots ,  \label{Lambda}
\end{equation}%
then
\begin{equation}
HV\left( f,T^{d}\right) \leq M\left( d\right) \Lambda ^{\#}V\left(
f,T^{d}\right) .  \label{embed}
\end{equation}
\end{theorem}

\begin{proof}
We have to prove that for any $\alpha :=\{j_{1},...,j_{p}\}\subset D$%
\begin{equation}  \label{main1}
\sup_{\{I_{i_{j}}^{j}\}\in \Omega }\sum_{i_{1},\ldots ,i_{p}}\frac{%
\left\vert f\left( I_{i_{1}}^{1}\times \cdots \times I_{i_{p}}^{p},x_{%
\widetilde{\alpha }}\right) \right\vert }{i_{1}\cdots i_{p}} \leq M\left(
d\right) \sum\limits_{s=1}^{d}\Lambda ^{\#}V_{s}\left( f, T^{d}\right)
\end{equation}
To this end, observe that
\begin{eqnarray}
&&\sum_{i_{1},\ldots ,\,i_{p}}\frac{\left\vert f\left( I_{i_{1}}^{1}\times
\cdots \times I_{i_{p}}^{p},x_{\widetilde{\alpha }}\right) \right\vert }{%
i_{1}\cdots i_{p}}  \label{sum} \\
&=&\sum_{\sigma }\sum_{i_{\sigma (1)}\leq \cdots \leq i_{\sigma (p)}}\frac{%
\left\vert f\left( I_{i_{1}}^{1}\times \cdots \times I_{i_{p}}^{p},x_{%
\widetilde{\alpha }}\right) \right\vert }{i_{1}\cdots i_{p}},  \notag
\end{eqnarray}%
where the sum is taken over all rearrangements $\sigma =\{\sigma
(k)\}_{k=1}^{p}$ of the set $\{1,2,\ldots ,p\}$.

Next, we have
\begin{eqnarray}
&&\sum_{i_{1}\leq \cdots \leq i_{p}}\frac{\left\vert f\left(
I_{i_{1}}^{1}\times \cdots \times I_{i_{p}}^{p},x_{\widetilde{\alpha }%
}\right) \right\vert }{i_{1}\cdots i_{p}}  \label{rear} \\
&=&\sum\limits_{i_{p}}\frac{1}{i_{p}}\sum_{i_{1}\leq \cdots \leq i_{p}}\frac{%
\left\vert f\left( I_{i_{1}}^{1}\times \cdots \times I_{i_{p}}^{p},x_{%
\widetilde{\alpha }}\right) \right\vert }{i_{1}\cdots i_{p-1}}.  \notag
\end{eqnarray}%
Taking into account that for the fixed $i_{p}\left( i_{1}\leq \cdots \leq
i_{p}\right) $ there exists $x_{1}^{i_{p}},\ldots ,x_{p-1}^{i_{p}}\in T$
such that
\begin{equation*}
{\left\vert f\left( I_{i_{1}}^{1}\times \cdots \times I_{i_{p}}^{p},x_{%
\widetilde{\alpha }}\right) \right\vert }\leq 2^{d}\left\vert f\left(
I_{i_{p}}^{p},x_{1}^{i_{p}},\ldots ,x_{p-1}^{i_{p}},x_{\widetilde{\alpha }%
}\right) \right\vert
\end{equation*}
from (\ref{rear}) we obtain%
\begin{eqnarray*}
&&\sum_{i_{1}\leq \cdots \leq i_{p}}\frac{\left\vert f\left(
I_{i_{1}}^{1}\times \cdots \times I_{i_{p}}^{p},x_{\widetilde{\alpha }%
}\right) \right\vert }{i_{1}\cdots i_{p}} \\
&\leq &2^{d}\sum\limits_{i_{p}}\frac{\left\vert f\left(
I_{i_{p}}^{p},x_{1}^{i_{p}},\ldots ,x_{p-1}^{i_{p}},x_{\widetilde{\alpha }%
}\right) \right\vert }{i_{p}}\sum_{i_{1}\leq \cdots \leq i_{p}}\frac{1}{%
i_{1}\cdots i_{p-1}} \\
&\leq &M\left( d\right) \sum\limits_{i_{p}}\frac{\log ^{d-1}\left(
i_{p}+1\right) }{i_{p}}\left\vert f\left( I_{i_{p}}^{p},x_{1}^{i_{p}},\ldots
,x_{p-1}^{i_{p}},x_{\widetilde{\alpha }}\right) \right\vert  \\
&\leq &M\left( d\right) \Lambda ^{\#}V_{i_{p}}\left( f,T^{d}\right) \leq
M\left( d\right) \Lambda ^{\#}V\left( f,T^{d}\right) .
\end{eqnarray*}%
Similarly one can obtain bounds for other summands in the right hind side of
(\ref{sum}), which imply (\ref{embed}). Theorem \ref{t1} is proved.
\end{proof}

\begin{corollary}
If the sequence $\Lambda$ is defined by (\ref{Lambda}), then $%
\Lambda^\#BV(T^d)\subset HBV(T^d)$.
\end{corollary}

Now, we denote
\begin{equation}  \label{Delta}
\Delta:=\{\delta=(\delta_1,\ldots,\delta_d):\delta_i=\pm1, \ i=1,2,\ldots,d\}
\end{equation}
and
\begin{equation*}
\pi_{\varepsilon\delta}(x):=(x_1,\,
x_1+\varepsilon\delta_1)\times\cdots\times (x_d,\, x_d+\varepsilon\delta_d),
\end{equation*}
for $x=(x_1,\ldots,x_d)\in R^d$ and $\varepsilon>0$. We set $%
\pi_{\delta}(x):=\pi_{\varepsilon\delta}(x)$, if $\varepsilon=1$.

For a function $f$ defined in some neighbourhood of a point $x$ and $%
\delta\in \Delta$ we set
\begin{equation}  \label{lim}
f_\delta(x):=\lim_{t\in \pi_\delta(x),\ t\to x} f(t),
\end{equation}
if the last limit exists.

\begin{theorem}
\label{th2} Suppose $f\in \Lambda ^{\#}BV\left( T^{d}\right)$ for some
sequence $\Lambda=\{\lambda_n\}$.

a) If the limit $f_\delta(x)$ exists for some $x=(x_1,\ldots,x_d)\in T^d$
and some $\delta=(\delta_1,\ldots,\delta_d)\in \Delta$, then
\begin{equation}  \label{th2a}
\lim\limits_{\varepsilon \rightarrow 0}\Lambda ^{\#}V\left(
f,\pi_{\varepsilon\delta}(x)\right) =0.
\end{equation}

b) If $f$ is continuous on some compact $K\subset T^{d}$, then
\begin{equation}  \label{th2b}
\lim_{\varepsilon \rightarrow 0}\Lambda ^{\#}V\left( f, \left[%
x_{1}-\varepsilon ,x_{1}+\varepsilon \right] \times\cdots\times \left[%
x_{d}-\varepsilon ,x_{d}+\varepsilon \right] \right) =0
\end{equation}
uniformly with respect to $x=(x_1,\ldots,x_d)\in K$.
\end{theorem}

\begin{proof}
According to Definition \ref{def5}, we need to prove that
\begin{equation}
\lim\limits_{\varepsilon \rightarrow 0}\Lambda ^{\#}V_{s}\left( f,\pi
_{\varepsilon \delta }(x)\right) =0  \label{th2as}
\end{equation}%
for any $s=1,2,\ldots ,d$. Without loss of generality we can assume that $%
s=1 $ and $\delta _{i}=1$ for $i=1,2,\ldots ,d$. Assume to the contrary that
(\ref{th2as}) does not holds:
\begin{equation*}
\lim\limits_{\varepsilon \rightarrow 0}\Lambda ^{\#}V_{1}\left( f,\pi
_{\varepsilon \delta }(x)\right) \neq 0.
\end{equation*}%
Then there exists a number $\alpha $ such that
\begin{equation}
\Lambda ^{\#}V_{1}\left( f,\pi _{\varepsilon \delta }(x)\right) >\alpha >0
\label{alpha}
\end{equation}%
for any $\varepsilon >0$.

Using induction on $k=1,2,\ldots$, we construct positive numbers $%
\varepsilon_k$ and the sequences of collections of non-overlapping intervals
\begin{equation}  \label{intervals}
I_{i}^{1}\subset \left(x_{1}+\varepsilon_{k+1},x_{1}+\varepsilon_k
\right),\quad i=n_k+1,...,n_{k+1}
\end{equation}
and vectors
\begin{equation}  \label{vectors}
\beta^ i=(\beta^i_1,\ldots,\beta^i_d)\in \pi_{\varepsilon_k\delta}(x),\quad
i=n_{k}+1,...,n_{k+1}
\end{equation}
as follows. By (\ref{alpha}), for a fixed number $\varepsilon_1>0$ we find a
collection of non-overlapping intervals
\begin{equation*}
I_{i}^{1}\subset \left(x_1,x_{1}+\varepsilon_1 \right),\quad i=1,...,n_{1}
\end{equation*}
and vectors
\begin{equation*}
\beta^ i=(\beta^i_1,\ldots,\beta^i_d)\in \pi_{\varepsilon_1\delta}(x),\quad
i=1,...,n_{1}
\end{equation*}
such that
\begin{equation}  \label{alpha1}
\sum\limits_{i=1}^{n_{1}}\frac{\left\vert f\left(
I_{i}^{1};\beta^i_2,\ldots,\beta^i_d\right) \right\vert }{\lambda _{i}}%
>\alpha.
\end{equation}

Now, suppose the number $\varepsilon _{k}$, intervals (\ref{intervals}) and
the vectors (\ref{vectors}) for some $k=1,2\ldots $ are constructed . Since
the limit $f_{\delta }(x)$ exists, we can choose $\varepsilon _{k+1}$
satisfying
\begin{equation}
0<\varepsilon _{k+1}<\varepsilon _{k},\qquad \left( x_{1},x_{1}+\varepsilon
_{k+1}\right) \bigcap \left( \bigcup\limits_{i=1}^{n_{k}}I_{i}^{1}\right)
=\emptyset  \label{intervals1}
\end{equation}%
and
\begin{equation}
\sum\limits_{i=1}^{n_{k}}\frac{\left\vert f\left( J_{i}^{1};\gamma
_{2}^{i},\ldots ,\gamma _{d}^{i}\right) \right\vert }{\lambda _{i}}<\frac{%
\alpha }{2}  \label{alpha1}
\end{equation}%
for any collection of non-overlapping intervals
\begin{equation*}
J_{i}^{1}\subset \left( x_{1},x_{1}+\varepsilon _{k+1}\right) ,\quad
i=1,...,n_{k}
\end{equation*}%
and for any vectors
\begin{equation*}
\gamma ^{i}=(\gamma _{1}^{i},\ldots ,\gamma _{d}^{i})\in \pi _{\varepsilon
_{k+1}\delta }(x),\quad i=1,...,n_{k}.
\end{equation*}%
Further, according to (\ref{alpha}) there is a collection of non-overlapping
intervals
\begin{equation}
J_{i}^{1}\subset \left( x_{1},x_{1}+\varepsilon _{k+1}\right) ,\quad
i=1,...,n_{k+1}  \label{intervals2}
\end{equation}%
and vectors
\begin{equation*}
\gamma ^{i}=(\gamma _{1}^{i},\ldots ,\gamma _{d}^{i})\in \pi _{\varepsilon
_{k+1}\delta }(x),\quad i=1,...,n_{k+1}
\end{equation*}%
such that
\begin{equation}
\sum\limits_{i=1}^{n_{k+1}}\frac{\left\vert f\left( J_{i}^{1};\gamma
_{2}^{i},\ldots ,\gamma _{d}^{i}\right) \right\vert }{\lambda _{i}}>\alpha .
\label{alpha2}
\end{equation}%
Now, denoting
\begin{equation}
I_{i}^{1}=J_{i}^{1},\beta ^{i}=\gamma ^{i}\quad \text{for}\quad
i=n_{k}+1,\ldots ,n_{k+1},  \label{intervals3}
\end{equation}%
from (\ref{alpha1}) and (\ref{alpha2}) we get
\begin{equation}
\sum\limits_{i=n_{k}+1}^{n_{k+1}}\frac{\left\vert f\left( I_{i}^{1};\beta
_{2}^{i},\ldots ,\beta _{d}^{i}\right) \right\vert }{\lambda _{i}}>\frac{%
\alpha }{2}.  \label{alpha3}
\end{equation}%
Intervals (\ref{intervals}) and vectors (\ref{vectors}) for $k=1,2,\ldots $,
are constructed.

By (\ref{intervals1}), (\ref{intervals2}) and (\ref{intervals3}), the
intervals $I^1_i$ are non-overlapping for $i=1,2,\ldots$, while according to
(\ref{alpha3}),
\begin{equation*}
\sum\limits_{i=1}^{\infty}\frac{\left\vert f\left(
I_{i}^{1};\beta^i_2,\ldots,\beta^i_d\right) \right\vert} {\lambda _{i}}%
=\infty.
\end{equation*}
Consequently, $\Lambda ^{\#}V_1\left( f,T^d\right)=\infty$. This
contradiction completes proof of the statement a) of Theorem \ref{th2}.

To prove statement b), observe that a) obviously implies (\ref{th2b}) for
any point $x\in T^d$, where $f$ is continuous. Hence, we have to prove that (%
\ref{th2b}) holds uniformly with respect to $x\in K$, provided that $f$ is
continuous on the compact $K\subset T^d$.

To this end let us assume to the contrary that (\ref{th2b}) does not hold
uniformly on $K$. Then there exist $\delta >0$ and sequences
\begin{equation*}
x^i=(x^i_{1},\ldots x^i_d)\in K \ \text{and}\ \varepsilon _{i}>0, \quad
i=1,2,\ldots \quad \text{with}\quad \varepsilon _{i}\rightarrow 0
\end{equation*}
such that%
\begin{equation*}
\Lambda ^{\#}V\left( f;\left[ x_{1}^{i }-\varepsilon _{i},x_{1}^{i
}+\varepsilon _{i}\right] \times\cdots\times \left[ x_{d}^{i
}-\varepsilon_{i},x_{d}^{i }+\varepsilon _{i}\right] \right) \geq \delta >0%
\text{.}
\end{equation*}%
Since $K$ is compact we can assume without loss of generality that $x^{ i
}\rightarrow x$ for some $x=(x_1,\ldots,x_d)\in K$. Then obviously for each $%
\varepsilon >0$ there is a number $i(\varepsilon)$ such that%
\begin{equation*}
\left[ x_{j}^{i }-\varepsilon _{i},x_{j}^{i }+\varepsilon _{i}\right]
\subset \left[ x_{j}-\varepsilon ,x_{j}+\varepsilon \right],\quad
j=1,\ldots,d \quad\text{for}\quad i>i\left( \varepsilon \right).
\end{equation*}%
Consequently,%
\begin{equation*}
\Lambda ^{\#}V\left( f;\left[ x_{1}-\varepsilon ,x_{1}+\varepsilon \right]
\times\cdots\times \left[ x_{d}-\varepsilon,x_{d}+\varepsilon\right]
\right)\geq \delta >0,
\end{equation*}%
for any $\varepsilon >0$, which is a contradiction.

Theorem \ref{th2} is proved.
\end{proof}

Next, we define
\begin{equation*}
v_{s}^{\#}\left( f,n\right) :=\sup_{\{x^{i}\}_{i=1}^{n}\subset
T^d}\sup\limits_{\{I_{i}^{s}\}_{i=1}^{n}\in \Omega
_{n}}\sum\limits_{i=1}^{n}\left\vert f\left( I_{i}^{s},x^i\{s\} \right)
\right\vert ,\quad s=1,...,d,\quad n=1,2,\ldots,
\end{equation*}
where $x^i\{s\}$ is as in (\ref{xis}). The following theorem holds.

\begin{theorem}
\label{v(n)}If the function $f(x),\ x\in T^d$ satisfies the condition
\begin{equation*}
\sum\limits_{n=1}^{\infty }\frac{v_{s}^{\#}\left( f,n\right) \log
^{d-1}\left( n+1\right) }{n^{2}}<\infty ,\qquad \ s=1,2,...,d,
\end{equation*}%
then $f\in \left\{ \frac{n}{\log ^{d-1}\left( n+1\right) }\right\}
^{\#}BV\left( T^{d}\right) .$
\end{theorem}

\begin{proof}
Let $s=1,\dots ,d$ be fixed. The for any collection of intervals ${%
\{I_{i}^{s}\}_{i=1}^{n}\in \Omega _{n}}$ and a sequence of vectors ${%
\{x^{i}\}_{i=1}^{n}\in T^{d}}$, using Abel's partial summation we obtain%
\begin{eqnarray}
&&\sum\limits_{j=1}^{n}\frac{\left\vert f\left( I_{j}^{s},x^{j}{\{s\}}%
\right) \right\vert \log ^{d-1}\left( j+1\right) }{j}  \label{v} \\
&=&\sum\limits_{j=1}^{n-1}\left( \frac{\log ^{d-1}\left( j+1\right) }{j}-%
\frac{\log ^{d-1}\left( j+2\right) }{j+1}\right)
\sum\limits_{k=1}^{j}\left\vert f\left( I_{k}^{s},x^{k}{\left\{ s\right\} }%
\right) \right\vert   \notag \\
&&+\frac{\log ^{d-1}\left( n+1\right) }{n}\sum\limits_{j=1}^{n}\left\vert
f\left( I_{j}^{s},x^{j}{\left\{ s\right\} }\right) \right\vert   \notag \\
&\leq &\sum\limits_{j=1}^{n-1}\left( \frac{\log ^{d-1}\left( j+1\right) }{j}-%
\frac{\log ^{d-1}\left( j+2\right) }{j+1}\right) v_{s}^{\#}\left( f,j\right)
\notag \\
&&+\frac{\log ^{d-1}\left( n+1\right) }{n}v_{s}^{\#}\left( f,n\right) .
\notag
\end{eqnarray}%
Using the inequality
\begin{eqnarray}
&&\frac{\log ^{d-1}\left( n+1\right) }{n}v_{s}^{\#}\left( f,n\right)
\label{v0} \\
&\leq &\sum\limits_{j=n}^{\infty }\left( \frac{\log ^{d-1}\left( j+1\right)
}{j}-\frac{\log ^{d-1}\left( j+2\right) }{j+1}\right) v_{s}^{\#}\left(
f,j\right) ,  \notag
\end{eqnarray}%
from (\ref{v}) we get%
\begin{equation}
\left\{ \frac{n}{\log ^{d-1}\left( n+1\right) }\right\} ^{\#}V_{s}\left(
f,T^{d}\right) \leq c\sum\limits_{n=1}^{\infty }\frac{v_{s}^{\#}\left(
f,n\right) \log ^{d-1}\left( n+1\right) }{n^{2}}<\infty .  \label{v1}
\end{equation}

Theorem \ref{v(n)} is proved.
\end{proof}

\section{\protect\medskip Convergence of multiple Fourier series}

We suppose throughout this section, that $T=[0,2\pi )$ and $T^{d}=[0,2\pi
)^{d}$, $d\ge2$, stands for the $d$-dimensional torus.

We denote by $C(T^{d})$ the space of continuous and $2\pi $-periodic with
respect to each variable functions with the norm
\begin{equation*}
\Vert f\Vert _{C}:=\sup_{\left( x_{1},\ldots ,\,x_{d}\right) \in
T^{d}}|f(x_{1},\ldots ,x_{d})|.
\end{equation*}

The Fourier series of the function $f\in L^{1}\left( T^{d}\right) $ with
respect to the trigonometric system is the series
\begin{equation*}
Sf\left( x_{1},...,x_{d}\right) :=\sum_{n_{1},...,n_{d}=-\infty }^{+\infty }%
\widehat{f}\left( n_{1},....,n_{d}\right) e^{i\left( n_{1}x_{1}+\cdots
+n_{d}x_{d}\right) },
\end{equation*}%
where
\begin{equation*}
\widehat{f}\left( n_{1},....,n_{d}\right) =\frac{1}{\left( 2\pi \right) ^{d}}%
\int_{T^{d}}f(x^{1},...,x^{d})e^{-i\left( n_{1}x_{1}+\cdots
+n_{d}x_{d}\right) }dx_{1}\cdots dx_{d}
\end{equation*}%
are the Fourier coefficients of $f$.

In this paper we consider convergence of \textbf{only rectangular partial
sums} (convergence in the sense of Pringsheim) of $d$-dimensional Fourier
series. Recall that the rectangular partial sums are defined as follows:
\begin{eqnarray*}
&&S_{N_{1},...,N_{d}}f\left( x_{1},...,x_{d}\right) \\
&&:=\sum_{n_{1}=-N_{1}}^{N_{1}}\cdots \sum_{n_{d}=-N_{d}}^{N_{d}}\widehat{f}%
\left( n_{1},....,n_{d}\right) e^{i\left( n_{1}x^{1}+\cdots
+n_{d}x^{d}\right) }.
\end{eqnarray*}

We say that the point $x\in T^d$ is \textit{a regular point} of a function $%
f $, if the limit $f_\delta(x)$ defined by (\ref{lim}) exists for any $%
\delta\in\Delta$ (see(\ref{Delta})). For the regular point $x$ we denote%
\begin{equation}
f^{\ast }\left(x\right) :=\frac{1}{2^{d}}\sum_ {\delta\in\Delta}
f_\delta\left(x\right) .  \label{limit}
\end{equation}

\begin{definition}
We say that the class of functions $V \subset L^{1}(T^{d})$ is a class of
convergence on $T^{d}$, if for any function $f\in V $

1) the Fourier series of $f$ converges to $f^{\ast }({x})$ at any regular
point ${x}\in T^{d}$,

2) the convergence is uniform on a compact $K\subset T^{d}$, if $f$ is
continuous on $K$.
\end{definition}

The well known Dirichlet-Jordan theorem (see \cite{Zy}) states that the
Fourier series of a function $f(x), \ x\in T$ of bounded variation converges
at every point $x$ to the value $\left[ f\left( x+0\right)
+f\left(x-0\right) \right] /2$. If $f$ is in addition continuous on $T$,
then the Fourier series converges uniformly on $T$.

Hardy \cite{Ha} generalized the Dirichlet-Jordan theorem to the double
Fourier series and proved that $BV$ is a class of convergence on $T^{2}$.

The following theorem was proved by Waterman (for $d=1$) and Sahakian (for $%
d=2$).

\begin{Sah}[Waterman \protect\cite{W}, Sahakian \protect\cite{Saha}]
If $d=1$ or $d=2$, then the class $HBV\left( T^{d}\right) $ is a class of
convergence on $T^{d}$.
\end{Sah}

In \cite{Bakh1} Bakhvalov proved that the class $HBV$ is not a class of
convergence on $T^{d}$, if $d>2$. On the other hand, he proved the following
theorem.

\begin{Bakh}[Bakhvalov \protect\cite{Bakh1}]
\label{B} The class $CHV\left( T^{d}\right) $ is a class of convergence on $%
T^{d}$ for any $d=1,2,\ldots $
\end{Bakh}

Convergence of spherical and other partial sums of double Fourier series of
functions of bounded $\Lambda $-variation was investigated in deatails by
Dyachenko \cite{D1,D2}.

In \cite{GogSah,GogSah2} Goginava and Sahakian investigated convergence of
multiple Fourier series of functions of bounded partial $\Lambda$-variation.
In particular, the following theorem was proved.

\begin{Goginava-Sahakian}
\bigskip a) If and $\Lambda =\left\{ \lambda _{n}\right\} _{n=1}^{\infty }$
with
\begin{equation*}
\lambda _{n}=\frac{n}{\log ^{d-1+\varepsilon }(n+1)},\qquad n=1,2,\ldots
,\quad d>1,
\end{equation*}%
for some $\varepsilon >0$, then the class $P\Lambda BV\left( T^{d}\right) $
is a class of convergence on $T^{d}$.\newline
\medskip b) If $\Lambda =\left\{ \lambda _{n}\right\} _{n=1}^{\infty }$ with
\begin{equation*}
\lambda _{n}=\frac{n}{\log ^{d-1}(n+1)},\qquad n=1,2,\ldots ,\quad d>1,
\end{equation*}%
then the class $P\Lambda BV\left( T^{d}\right) $ is not a class of
convergence on $T^{d}$.
\end{Goginava-Sahakian}

In \cite{DW}, Dyachenko and Waterman proved that the class $\Lambda ^{\ast
}BV(T^{2})$ is a class convergence on $T^{2}$ for $\Lambda =\{\lambda _{n}\}$
with $\lambda _{n}=\frac{n}{\ln \left( n+1\right) }$, $n=1,2,\ldots $

The main result of the present paper is the following theorem.

\begin{theorem}
\label{main} a) If $\Lambda =\left\{ \lambda _{n}\right\} _{n=1}^{\infty }$
with
\begin{equation}  \label{Lambdaa}
\lambda _{n}=\frac{n}{\log ^{d-1}(n+1)},\qquad n=1,2,\ldots,\quad d>1,
\end{equation}%
then the class $\Lambda ^{\#}BV\left( T^{d}\right) $ is a class of
convergence on $T^d$.

b)If $\Lambda =\left\{ \lambda _{n}\right\} _{n=1}^{\infty }$ with
\begin{equation}  \label{Lambdab}
\lambda_n :=\left\{ \frac{n\xi _{n}}{\log ^{d-1}\left( n+1\right) }\right\}
,\quad n=1,2,\ldots\quad d>1,
\end{equation}
where $\xi _{n}\to \infty$ as $n\rightarrow \infty$, then there exists a
continuous function $f\in \Lambda ^{\#}BV\left( T^{d}\right) $ such that the
cubical partial sums of $d$-dimensional Fourier series of $f$ diverge
unboundedly at $\left( 0,...,0\right) \in T^{d}$.
\end{theorem}

\begin{proof}[Proof of Theorem \protect\ref{main}]
The proof of the part a) is based on the following statement, that in the
case $d=2$ is proved by Sahakian (see formulaes (33) and (35) in \cite{Saha}%
). For an arbitrary $d>2$ the proof is similar.

\begin{Sahakian}
Suppose $f\in HV\left( T^{d}\right)$ and $x\in T^d$. If the limit $%
f_\delta(x)$ exists for any $\delta\in \Delta$, then for any $\varepsilon >0$
\begin{equation*}
\left\vert S_{n_{1},...,n_{d}}f\left( x\right) -f^\ast\left( x\right)
\right\vert \leq M\left( d\right) \sum_{\delta \in \Delta}HV\left(
f;\pi_{\varepsilon\delta}(x) \right) +o\left( 1\right),
\end{equation*}
as $n_i\to\infty,\ i=1,2,\ldots,d$.

Moreover, the quantity $o(1)$ tends to $0$ uniformly on a compact $K$, if $f$
is continuous on $K$.
\end{Sahakian}

Now, if the sequence $\Lambda=\{\lambda_n\}$ is defined by (\ref{Lambdaa})
and $f\in \Lambda^{\#}BV\left( T^{d}\right)$, then Lemma S and Theorem \ref%
{t1} imply that for any $\varepsilon >0$
\begin{equation}  \label{ps}
\left\vert S_{n_{1},...,n_{d}}f\left( x\right) -f^\ast\left( x\right)
\right\vert \leq M\left( d\right) \sum_{\delta \in \Delta}\Lambda
^{\#}V\left( f;\pi_{\varepsilon\delta} (x)\right) +o\left(1\right),
\end{equation}
which combined with Theorem \ref{th2} completes the proof of a).

To prove part b) suppose that $\Lambda =\{\lambda _{n}\}$ is a sequence
defined by (\ref{Lambdab}). It is not hard to see that the class $%
C(T^{d})\cap \Lambda ^{\#}BV\left(T^{d}\right) $ is a Banach space with the
norm
\begin{equation*}
\Vert f\Vert _{\Lambda ^{\#}BV}:=\Vert f\Vert _{C}+\Lambda ^{\#}BV(f).
\end{equation*}
Denoting
\begin{equation*}
A_{i_{1},\ldots ,i_{d}}:=\left[ \frac{\pi i_{1}}{N+1/2},\frac{\pi \left(
i_{1}+1\right) }{N+1/2}\right) \times \cdots \times \left[ \frac{\pi i_{d}}{%
N+1/2},\frac{\pi \left( i_{d}+1\right) }{N+1/2}\right),
\end{equation*}%
we consider the following functions
\begin{equation*}
g_{N}\left( x_{1},\ldots ,x_{d}\right)
:=\sum\limits_{i_{1},...,i_{d}=1}^{N-1}1_{A_{i_{1},\ldots ,i_{d}}}\left(
x_{1},\ldots ,x_{d}\right) \prod\limits_{s=1}^{d}\sin \left( N+1/2\right)
x_{s},
\end{equation*}%
for $N=2,3,\ldots$, where $1_{A}\left( x_{1},\ldots ,x_{d}\right) $ is the
characteristic function of a set $A\subset T^{d}$.

It is easy to check that%
\begin{equation*}
\left\{ \frac{n\xi _{n}}{\log ^{d-1}\left( n+1\right) }\right\}
^{\#}V_{s}\left( g_{N}\right) \leq c\sum\limits_{i=1}^{N-1}\frac{\log
^{d-1}\left( i+1\right) }{i\xi _{i}}=o\left( \log ^{d}N\right)
\end{equation*}%
and hence
\begin{equation*}
\left\Vert g_{N}\right\Vert _{\Lambda ^{\#}BV}=o\left( \log ^{d}N\right)
=\eta _{N}\log ^{d}N,
\end{equation*}%
where $\eta _{N}\rightarrow 0$ as $N\rightarrow \infty $. Now, setting
\begin{equation*}
f_{N}:=\frac{g_{N}}{\eta _{N}\log ^{d}N},\quad N=2,3,\ldots ,
\end{equation*}%
we obtain that $f_{N}\in \Lambda ^{\#}BV\left( T^{d}\right) $ and
\begin{equation}
\sup\limits_{N}\left\Vert f_{N}\right\Vert _{\Lambda ^{\#}BV}<\infty .
\label{norm}
\end{equation}%
Now, for the cubical partial sums of the d-dimensional Fourier series of $%
f_{N}$ at $\left( 0,...,0\right) \in T^{d}$ we have that
\begin{eqnarray}
&&\pi ^{d}S_{N,\cdots ,N}f_{N}\left( 0,\cdots ,0\right)  \label{low} \\
&=&\frac{1}{\eta _{N}\log ^{d}N}\sum\limits_{i_{1},...,i_{d}=1}^{N-1}\int%
\limits_{A_{i_{1},\cdots ,\,i_{d}}}\prod\limits_{s=1}^{d}\frac{\sin
^{2}\left( N+1/2\right) x_{s}}{2\sin \left( x_{s}/2\right) }dx_{1}\cdots
dx_{d}  \notag \\
&\geq &\frac{c}{\eta _{N}\log ^{d}N}\sum\limits_{i_{1},...,i_{d}=1}^{N-1}%
\frac{1}{i_{1}\cdots i_{d}}\geq \frac{c}{\eta _{N}}\rightarrow \infty  \notag
\end{eqnarray}%
as $N\rightarrow \infty $. Applying the Banach-Steinhaus Theorem, from (\ref%
{norm}) and (\ref{low}) we conclude that there exists a continuous function $%
f\in \Lambda ^{\#}BV\left( T^{d}\right) $ such that
\begin{equation*}
\sup_{N}|S_{N,\cdots ,N}f(0,\cdots ,0)|=\infty .
\end{equation*}%
Theorem \ref{main} is proved.
\end{proof}

The next theorem follows from Theorems \ref{v(n)} and \ref{main}.

\begin{theorem}
For any $d>1$ the class of functions $f(x),\ x\in T^d$ satisfying the
following condition
\begin{equation*}
\sum\limits_{n=1}^{\infty }\frac{v_{s}^{\#}\left( f,n\right) \log
^{d-1}\left( n+1\right) }{n^{2}}<\infty ,~\ \ s=1,...,d,
\end{equation*}
is a class of convergence.
\end{theorem}

\end{document}